\documentclass[12pt,reqno]{amsart}

\usepackage{fullpage}

\usepackage{times}
\usepackage[T1]{fontenc}
\usepackage{mathrsfs}
\usepackage{latexsym}
\usepackage[dvips]{graphics}
\usepackage[dvips]{graphicx}
\usepackage{epsfig}
\usepackage{amsmath,amsfonts,amsthm,amssymb,amscd}
\input amssym.def
\input amssym.tex
\usepackage{color}
\usepackage{hyperref}
\usepackage{url}

\usepackage{verbatim}

\newcommand{\ncr}[2]{\left({#1 \atop #2}\right)}
\newcommand{\twocase}[5]{#1 \begin{cases} #2 & \text{{\rm #3}}\\ #4
&\text{{\rm #5}} \end{cases}  }

\newcommand{\intii}{\int_{-\infty}^\infty}

\newtheorem{thm}{Theorem}[section]

\newtheorem{cor}[thm]{Corollary}
\newtheorem{lem}[thm]{Lemma}

\newtheorem{defi}[thm]{Definition}

\numberwithin{equation}{section}

\newcommand\be{\begin{equation}}
\newcommand\ee{\end{equation}}
\newcommand\bea{\begin{eqnarray}}
\newcommand\eea{\end{eqnarray}}
\newcommand\bi{\begin{itemize}}
\newcommand\ei{\end{itemize}}
\newcommand\ben{\begin{enumerate}}
\newcommand\een{\end{enumerate}}

\renewcommand{\geq}{\geqslant}
\renewcommand{\leq}{\leqslant}
\renewcommand{\mod}[1]{
	{\ifmmode\text{\rm\ (mod~$#1$)}\else
		\discretionary{}{}{\hbox{ }}\rm(mod~$#1$)\fi}
}

\newcommand{\R}{\ensuremath{\mathbb{R}}}

\newcommand{\N}{\ensuremath{\mathbb{N}}}

\newcommand{\E}{\ensuremath{\mathbb{E}}}
\newcommand{\w}{\ensuremath{\mathcal{W}}}
\newcommand{\muwd}{\mu_{\ensuremath{\mathcal{W}}_d}}

\newcommand{\capp}{\textsc{capp}}
\newcommand{\capps}{\textsc{capp}s}
\newsymbol\dnd 232D

\begin{document}

\title{On the spectral distribution of large weighted random regular graphs}

\author{Leo Goldmakher}
\address{
Department of Mathematics,
University of Toronto,
Toronto, ON,
Canada}
\email{\textcolor{blue}{\href{mailto:lgoldmak@math.toronto.edu}{lgoldmak@math.toronto.edu}}}

\author{Cap Khoury}
\address{
Department of Mathematics,
University of Michigan,
Ann Arbor, MI,
USA
}
\email{\textcolor{blue}{\href{mailto:prof.cap.khoury@gmail.com}{prof.cap.khoury@gmail.com}}}

\author{Steven J. Miller}
\address{
Department of Mathematics and Statistics,
Williams College,
Williamstown, MA,
USA
}
\email{\textcolor{blue}{\href{mailto:sjm1@williams.edu, Steven.Miller.MC.96@aya.yale.edu}{sjm1@williams.edu, Steven.Miller.MC.96@aya.yale.edu}}}

\author{Kesinee Ninsuwan}
\address{
Institute for Computational and Mathematical Engineering,
Stanford University,
Stanford, CA,
USA}
\email{\textcolor{blue}{\href{mailto:eveve@stanford.edu}{eveve@stanford.edu}}}

\begin{abstract}
McKay proved that the limiting spectral measures of the ensembles of $d$-regular graphs with $N$ vertices converge to Kesten's measure as $N\to\infty$. In this paper we explore the case of weighted graphs. More precisely, given a large $d$-regular graph we assign random weights, drawn from some distribution $\mathcal{W}$, to its edges. We study the relationship between $\mathcal{W}$ and the associated limiting spectral distribution obtained by averaging over the weighted graphs. Among other results, we establish the existence of a unique `eigendistribution', i.e., a weight distribution $\mathcal{W}$ such that the associated limiting spectral distribution is a rescaling of $\mathcal{W}$. Initial investigations suggested that the eigendistribution was the semi-circle distribution, which by Wigner's Law is the limiting spectral measure for real symmetric matrices. We prove this is not the case, though the deviation between the eigendistribution and the semi-circular density is small (the first seven moments agree, and the difference in each higher moment is $O(1/d^2)$). Our analysis uses combinatorial results about closed acyclic walks in large trees, which may be of independent interest.

\end{abstract}



\date{}

\thanks{Portions of this work were completed at REUs at AIM and Williams College; we thank our colleagues there for helpful conversations on earlier drafts. The first named author was partially supported by an NSERC Discovery Grant. The third named author was partially supported by NSF grants DMS0600848 and DMS0970067. The fourth named author was partially supported by NSF grant DMS0850577, Brown University and Williams College.}

\maketitle

\thispagestyle{empty}

\tableofcontents



\section{Introduction}

The eigenvalues of the adjacency matrices associated to graphs encode a wealth of information about the original graph, and are thus a natural and important object to study and understand. We consider $d$-regular graphs below. Thus $d$ is always an eigenvalue of the adjacency matrix; moreover, it is the largest eigenvalue in absolute value.  The simplest application of the eigenvalues is to determine whether or not a graph is connected, which happens if and only if $d$ is a simple eigenvalue. Our next application depends on the difference between the second largest (in absolute value) eigenvalue and $d$; this is called the spectral gap. A large spectral gap implies many desirable properties for the graph. Such graphs are well-connected, which means that we can have a graph with very few edges but all vertices able to communicate with each other very quickly. These graphs arise in communication network theory, allowing the construction of superconcentrators and non-blocking networks \cite{Bien, Pi}, and in coding theory \cite{SS} and cryptography \cite{GILVZ}. Alon \cite{Al} conjectured that as $N\to\infty$, for $d\ge 3$ and any $\epsilon > 0$, ``most'' $d$-regular graphs on $N$ vertices have their second largest (in absolute value) eigenvalue at most $2 \sqrt{d-1} + \epsilon$; it is known that the $2\sqrt{d-1}$ cannot be improved upon. Thanks to the work of Friedman \cite{Fr1, Fr2, Fr3} this is now a theorem, though the finer behavior around this critical threshold is still open (see \cite{MNS} for numerics and conjectures). For some basics of graph theory and constructions of families of expanders (graphs with a large spectral gap and thus good connectivity properties), see \cite{DSV, LPS, Mar, Sar1, Sar2}.

After investigating the largest two eigenvalue and their consequences, it is natural to continue and study the rest of the spectrum. Thirty years ago, McKay \cite{McK} investigated the distribution of eigenvalues of large, random $d$-regular graphs; \textbf{\emph{we always assume our graphs do not contain any self-loops or multiple edges}}. Under the assumption that the number of cycles is small relative to the size of the graph (which is true for most $d$-regular graphs as the number of vertices grows), he proved the existence of a limiting spectral distribution $\nu_d$ depending only on $d$, and gave an explicit formula for $\nu_d$. Moreover, he showed that if one renormalizes $\nu_d$ so that its associated density function has support $[-1,1]$, then the sequence of renormalized measures converges to Wigner's semicircle measure as $d \to \infty$. The goal of the present paper is to explore the more complicated situation for randomly weighted graphs. We weigh the graphs by attaching weights to each edge. There is an extensive literature on properties of weighted graphs (where we may weight either the edges or the graphs in the family); see \cite{ALHM, AL, Bo1, Bo2, ES, Ga, McD1, McD2, Po}
and the references therein for some results and applications.

More precisely, suppose $\w$ is a random variable with finite moments on $\R$ and density $p_\w$, and $G \in \mathcal{R}_{N,d}$, the set of simple $d$-regular graphs on $N$ vertices with no self-loops. We weigh each edge by independent identically distributed random variables (iidrv's) drawn from $\w$. In other words, we replace all nonzero entries in the adjacency matrix of $G$ by iidrv's drawn from $\w$; this is the same as taking the Hadamard product of a real symmetric weight matrix with the graph's adjacency matrix. Denote the spectrum of the weighted graph $G$ by
$\{\lambda_1 \leq \lambda_2 \leq \cdots \leq \lambda_N\}$, and consider the uniform measure $\nu_{d,G,\w}$ on this spectrum:
\begin{equation}
\label{eq:SpectralDistr}
\nu_{d,G,\w}(x) \ = \  \frac{1}{N} \#\big\{j \leq N : \lambda_j = x\big\}.
\end{equation}
As indicated by the subscripts, this measure depends on $d$, $G$, and $\w$. We are interested in the limiting behavior, so rather than focusing on any particular graph $G$ we take a sequence of graphs of increasing size. We first set some notation.\\

\begin{itemize}

\item $\mathcal{R}_{N,d}$: The set of simple $d$-regular graphs on $N$ vertices without self-loops.

\item $|G|$, $a_{ij}$: For a graph $G$, $|G|$ denotes the number of its vertices, and $a_{ij} = 1$ if vertices $i$ and $j$ are connected by an edge of $G$, and 0 otherwise.

\item $n_{{\rm cyl}}(k;G)$: The number of cycles of length $k$ in $G$.

\item $c_m$: We set $c_m$ to be the $m$\textsuperscript{th} moment of the semi-circle distribution, normalized to have variance 1/4. Thus $c_{2k+1} = 0$ and $c_{2k} = \frac{1}{4^k(k+1)} \ncr{2k}{k}$ (with $\frac1{k+1}\ncr{2k}{k}$ the $k$\textsuperscript{th} Catalan number).

\item $\mu_\mathcal{X}(k)$, $p_{\mathcal{X}}$, $\textbf{x}$: For $\mathcal{X}$ a random variable whose density has finite moments, $\mu_\mathcal{X}(k)$ is its $k$\textsuperscript{th} moment and $p_{\mathcal{X}}$ is the density associated to $\mathcal{X}$. Finally, $\textbf{x}$ is either an $N(N+1)/2$ vector (or, equivalently, an $N\times N$ real symmetric matrix) of independent random variables $x_{ij}$ drawn from $X$. We typically take $X$ to be our weight random variable $\w$.

\item $G_{\textbf{w}}$, $\mu_{d,\w}(k;G), \mu_{d,\w}(k)$: For a fixed $d$, weight distribution $\mathcal{W}$ and graph $G$, $G_{\textbf{w}}$ denotes the graph obtained by weighting the edges of $G$ by $\textbf{w}$, $\mu_{d,\w}(k;G)$ is the average (over weights drawn from $\w$) $k$\textsuperscript{th} moment of the associated spectral measures $\nu_{d,G_{\textbf{w}}}$, while $\mu_{d,\w}(k)$ is the average of $\mu_{d,\w}(k;G)$ over $G \in \mathcal{R}_{N,d}$.

\end{itemize}

\ \\

The following result is the starting point of our investigations. The unweighted case is due to McKay \cite{McK}; the existence proof in the general case proceeds similarly.

\begin{thm}
\label{thm:LimitDistr}
For any sequence of $d$-regular graphs $\{G_i\}$ such that $|G_i| \to \infty$ and
$n_{{\rm cyl}}(G_i) = o(|G_i|)$ for every $k \geq 3$, the limiting distribution
\be
\nu_{d,\w}(x)\ :=\ \lim_{i \to \infty} \nu_{d,G_i,\w}(x)
\ee
exists and depends only on $d$ and $\w$. In the unweighted case (i.e., each weight is 1) the density is
given by Kesten's measure: \be \twocase{f(x) \ = \ }{\frac{d}{2\pi(d^2-x^2)} \sqrt{4(d-1) - x^2},}{$|x| \le
2\sqrt{d-1}$}{0}{otherwise.} \ee Note that as $d\to\infty$, Kesten's measure tends to the semi-circle distribution.
\end{thm}

The difficulty is deriving good, closed-form expressions when the weights are non-trivial. To this end, we study the one-parameter family of maps
\begin{equation}
\label{eq:DefnOfTd}
T_d : \w \longmapsto \nu_{d,\w}.
\end{equation}
To understand the behavior of $T_d$, we investigate its eigendistributions, a concept we now explain. Recall that any measure $\nu$ can be rescaled by a real $\lambda > 0$ to form a new measure $\nu^{(\lambda)}$ by setting
\be\label{eq:effectoflambdaonmeasure}
\nu^{(\lambda)}(A) \ := \  \nu(\lambda A) \ \ \ \ \ \ ({\rm for\ all}\ A \subseteq{\R}).
\ee
If a distribution $\w$ satisfies
\be
T_d \w = \w^{(\lambda)}
\ee
for some $\lambda > 0$, we say $\w$ is an eigendistribution of $T_d$ with eigenvalue $\lambda$. We prove in \S\ref{sec:eigendistr} that for each $d$ the map $T_d$ has a unique eigendistribution, up to rescaling; this existence proof is a straightforward application of standard techniques.

Thus the natural question is to determine the eigendistribution for each $d$. Explicit formulas exist for the moments, but quickly become very involved. Brute force computations show that the first seven moments of the eigendistribution agree with the moments of a semi-circular distribution, suggesting that the semi-circle is the answer. If true this is quite interesting, as the semi-circle is the limiting spectral measure for real symmetric matrices (Wigner's law); moreover, as $d\to\infty$ the limiting spectral measure of the unweighted ensemble of $d$-regular graphs converges to the semi-circle. In fact, the motivation for this research was the following question: What weights must be introduced so that the weighted ensemble has the semi-circle as its density?

\begin{figure}
\begin{center}
\scalebox{.837}{\includegraphics{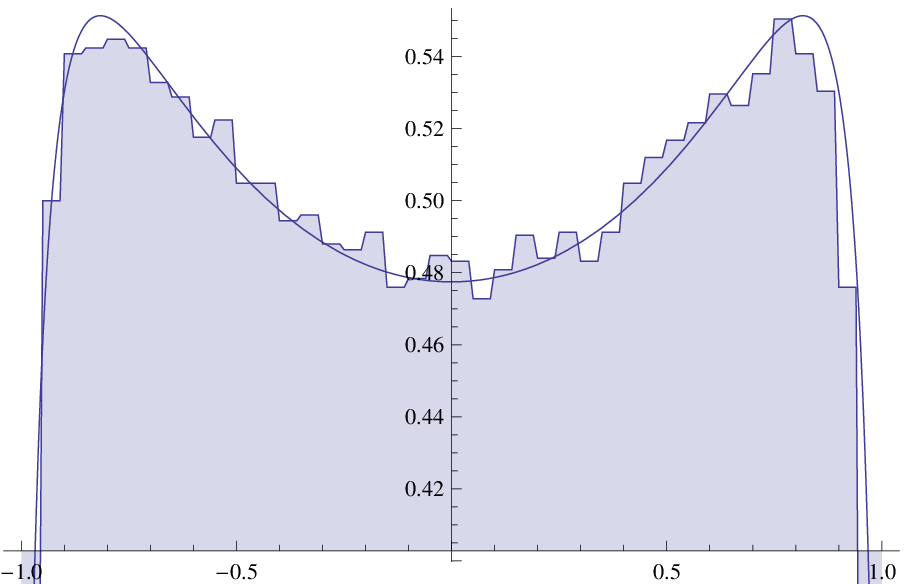}}\ \scalebox{.837}{\includegraphics{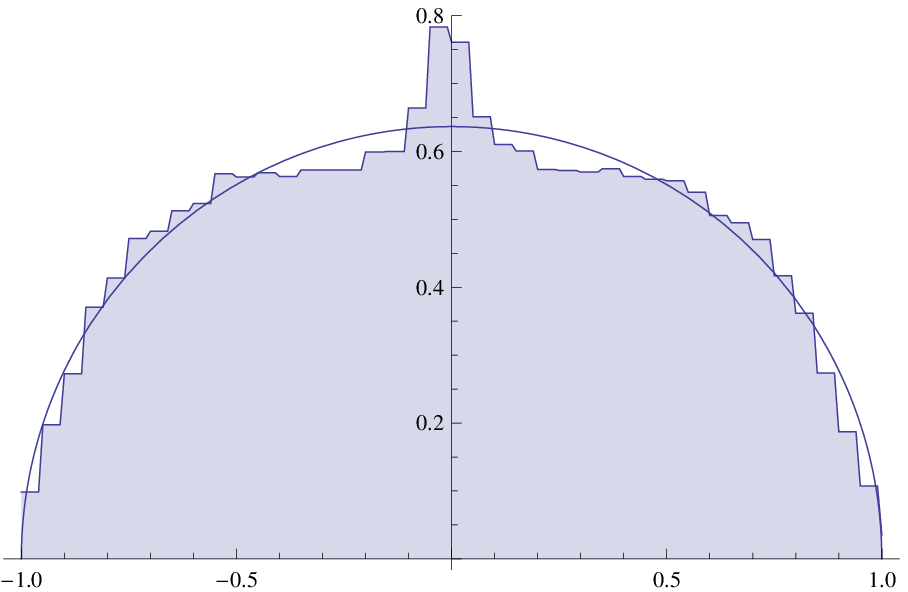}}
\caption{\label{fig:semicircleevidence} Numerical evidence `supporting' the semi-circular conjecture; here we have 100 matrices that are 4-regular and $200 \times 200$. The left plot is the density of eigenvalues in the unweighted case compared to Kesten's measure, while the right plot compares the density of eigenvalues with semi-circular weights to the semi-circular distribution.}
\end{center}\end{figure}

While a determination of the first few moments and numerical investigations (see Figure \ref{fig:semicircleevidence}) seemed to support the semi-circle as the eigendistribution, this conjecture is false, though the two distributions are close and agree as $d\to\infty$. For another ensemble where numerical data and heuristic arguments suggested a specific limiting spectral measure which was close to but not equal to the answer, see the work on real symmetric Toeplitz matrices \cite{BDJ,HM}.

To state our results precisely, we switch to the language of moments. In \S\ref{sec:combprelim} we define our notation, which we use to give a precise relationship between the moments of $\w$ and $T_d \w$ in terms of \emph{closed acyclic path patterns}, a combinatorial notion we develop in \S\ref{sect:capp}. From this we deduce our main result.

\begin{thm}\label{thm:mainresultuniquewdmoments} There is a unique eigendistribution of $T_d$ which has second moment equal to $1/4$, which we denote $\w_d$. Let $\mu_{\w_d}(k)$ denote the $k$\textsuperscript{${\rm th}$} moment of $\w_d$. Then for all $k \in \N$ we have $\mu_{\w_d}(2k+1) = 0$ and
\be
\mu_{\w_d}(2k)\ = \ c_{2k} + O\left(1/d^2\right),
\ee
where $c_m$ is the $m$\textsuperscript{${\rm th}$} moment of the semi-circle distribution normalized to have second moment 1/4.

We have $\mu_{\w_d}(2) = 1/4$, $\mu_{\w_d}(4) = 1/8$, $\mu_{\w_d}(6) = 5/64$ (all agreeing with the normalized semi-circular density), but
\be
\mu_{\w_d}(8) \ = \ \frac{7}{128} + \frac1{128(d^2+d+1)},
\ee
which disagrees with the eighth moment of the semi-circle, $7/128$.
\end{thm}

The difference in the eighth moments show that our error term is optimal. The fact that the error decays like $1/d^2$, as opposed to $1/d$, is the consequence of a beautiful combinatorial alignment which we describe in Lemma \ref{lem:Serendipity}.


In this paper we concentrate on deriving results about the eigendistribution $\w_d$ and not on the convergence of the individual weighted spectral measures to the average, as the techniques from \cite{McK} and standard arguments (see for example the convergence arguments in \cite{HM} or the method of moment arguments in \cite{Bi,Ta}) suffice to prove such convergence. We only quoted part of Theorem 1.1 of \cite{McK}; the rest of it refers to convergence of the corresponding cumulative distribution functions for graphs satisfying the two conditions in the theorem, and his argument applies with trivial modifications in our setting.

Though we do not examine it here, another natural avenue one could explore is the distribution of gaps between adjacent, normalized eigenvalues. This was studied in \cite{JMRR} for $d$-regular graphs. Their numerics support a GOE spacing law, which also governs the behavior for the eigenvalues of the ensemble of real symmetric matrices, but we are far from having a proof in this setting. The distribution of gaps is significantly harder than the density of eigenvalues, and it was only recently (see \cite{ERSY,ESY,TV1,TV2}) where these spacing measures were determined for non-Gaussian random matrix ensembles. There is now a large body of work on the density of eigenvalues and the gaps between them for different structured random matrix ensembles; see \cite{FM,Fo,Meh} for a partial history and the general theory, and \cite{BLMST,BCG,BHS1,BHS2,HM,KKMSX} and their references for some results on structured ensembles.



\section{Combinatorial Preliminaries}\label{sec:combprelim}

In this section we expand upon the ideas in the introduction, and develop some appropriate combinatorial notions. In particular, we introduce closed acyclic path patterns, which play a crucial role in our work.

We begin by formalizing the notion of a randomly weighted graph. Suppose as before that $G \in \mathcal{R}_{N,d}$ has adjacency matrix ${A= \big(a_{ij}\big)}$, and let $\w$ be a random variable whose probability density has finite moments. Let $\textbf{w} = \big\{w_{ij} : 1 \leq i \leq j \leq N\big\}$ denote a set of independent random variables drawn from $\w$, and form an $N \times N$ matrix ${A_\textbf{w} = \big(b_{ij}\big)}$, where
\be
b_{ij}\ = \ \begin{cases}
w_{ij} a_{ij} & \quad \text{if } i \leq j \\
w_{ji} a_{ji} & \quad \text{otherwise.}
\end{cases}
\ee
Observe that $A_\textbf{w}$ is a real symmetric matrix, and that $b_{nn} = 0$ for all $n$.  We may therefore interpret $A_\textbf{w}$ as the adjacency matrix of a weighted graph $G_\textbf{w}$ whose edges are weighted by the random variables $\textbf{w}$; equivalently, $G_{\textbf{w}}$ is the Hadamard product of our weight matrix and $G$'s adjacency matrix. We also note that at most $dN / 2$ of the entries $b_{ij}$ are nonzero.

We are interested in the relationship between the distribution $\w$ and the corresponding spectral distribution. Denote the eigenvalues of $A_\textbf{w}$ by ${\lambda_1 \leq \lambda_2 \leq \cdots \leq \lambda_N}$, and let $\nu_{d,G_\textbf{w}}$  be the uniform measure on this spectrum, as in \eqref{eq:SpectralDistr}; its density is thus \be {\rm d}\nu_{d,G_\textbf{w}}(x) \ = \ \frac1{N} \sum_{n=1}^N \delta(x - \lambda_n){\rm d}x, \ee where $\delta(u)$ is the Dirac delta functional.\footnote{We write $d$ for the degree of regularity, and ${\rm d}$ for differentials.} While we do not need the subscript $d$ as $G_{\textbf{w}}$ implicitly encodes the degree of regularity $d$, we prefer to be explicit and highlight the role of this important parameter. By definition and the eigenvalue trace formula, the $k^\text{th}$ moment $\mu_{\nu_{d,G_{\textbf{w}}}}(k)$ of the spectral distribution is
\begin{equation}
\label{eq:moment1}
\mu_{\nu_{d,G_{\textbf{w}}}}(k)
\ = \  \intii x^k {\rm d}\nu_{d, G_\textbf{w}}(x)
\ = \  \frac{1}{N} \sum_{n = 1}^{N} \lambda_n^k
\ = \  \frac{1}{N} {\rm Tr}(A_\textbf{w}^k);
\end{equation} we write $\mu_{\nu_{d,G_{\textbf{w}}}}(k)$ to emphasize that the regularity $d$ is fixed and we are studying a specific weighted graph $G_\textbf{w}$.

The following approach is standard and allows us to convert information on the matrix elements of $A_{\textbf{w}}$ (which we know) to information on the eigenvalues (which we desire). We have
\begin{equation}
\label{eq:moment2}
{\rm Tr}(A_\textbf{w}^k)\ = \ 	\sum_{i_1=1}^{N} \sum_{i_2=1}^{N} \cdots \sum_{i_k=1}^{N}
		b_{i_1 i_2} b_{i_2 i_3} \cdots b_{i_k i_1}.
\end{equation}
Thus we see that the $k^\text{th}$ moment of the spectral distribution associated to $G_\textbf{w}$ is the average weight of a closed walk of length $k$ in $G$ (where by the weight of a walk we mean the product of the weights of all edges traversed, counted with multiplicity).

Since we are interested in the dependence on the distribution $\w$, not on the specific values of the $N(N+1)/2$ random variables $\textbf{w} = (w_{ij})_{1 \le i,j \le N}$, we average over $\textbf{w}$ drawn from $\w$'s density $p_\w$ to obtain the `typical' $k^\text{th}$ moment $\mu_{d,\w}(k;G)$ of the weighted spectral distributions:
\be
\mu_{d,\w}(k;G) \ := \
	\intii \mu_{\nu_{d,G_{\textbf{w}}}}(k) {\rm d}\textbf{w}
	\	=\ \intii \cdots \intii \mu_{\nu_{d,G_{\textbf{w}}}}(k) \prod_{1 \le i \le j \le N} p_\w(w_{ij}){\rm d}w_{ij},
\ee
where $p_\w$ is the density function corresponding to distribution $\w$.

To build intuition for the later calculations, we calculate the first and second moments.

\begin{lem}[First Two Moments]\label{lem:firsttwoments} Fix $d$, $G \in \mathcal{R}_{N,d}$ and $\w$. We have $\mu_{d,\w}(1;G) = 0$ and $\mu_{d,\w}(2;G) =  d \mu_{\w}(2)$, where $\mu_{\w}(k)$ is the $k$\textsuperscript{${\rm th}$} moment of $\w$. Thus $\mu_{d,\w}(1) = 0$ and $\mu_{d,\w}(2) = d \mu_{\w}(2)$.
\end{lem}

\begin{proof} Since $b_{nn} = 0$ for all $n$, we see that
\be
\mu_{d,\w}(1;G)
\ =\ \intii \cdots \intii \mu_{\nu_{d,G_{\textbf{w}}}}(1) {\rm d}\textbf{w}
\ =\ \intii \cdots \intii \frac{1}{N} \sum_{n=1}^{N} b_{nn} \prod_{1 \le i \le j \le N} p_\w(w_{ij}){\rm d}w_{ij}
\ = \ 0.
\ee

To compute the second moment, we use $G$ is $d$-regular, $b_{ij} = a_{ij} w_{ij}$, and $b_{nn} = 0$ and $b_{ij} = b_{ji}$. Note the number of non-zero $a_{ij}$ is $dN/2$ (each vertex has $d$ edges emanating from it, and each edge is doubly counted), and recall $\mu_\w(2)$ denotes the second moment of the weight distribution $\w$. We obtain
\bea
\mu_{d,\w}(2;G)
&\ =\ & \intii \cdots \intii \mu_{\nu_{d,G_{\textbf{w}}}}(2) {\rm d}\textbf{w}
\ =\ \intii \cdots \intii \frac{2}{N} \sum_{1 \leq i < j \leq N} b_{ij}^2 p_\w(w_{ij}){\rm d}w_{ij}\nonumber\\
&\ =\ & \frac{2}{N}
		\sum_{\substack{1 \leq i < j \leq N \\ a_{ij} = 1}}
			\intii w_{ij}^2p_\w(w_{ij}){\rm d}w_{ij}
 \ =\ \frac{2}{N}
		\sum_{\substack{1 \leq i < j \leq N \\ a_{ij} = 1}} \mu_\w(2) \nonumber\\
&\ = \ & \frac{2}{N} \cdot \frac{dN}{2} \mu_\w(2) \ = \ d \mu_w(2).
\eea
\end{proof}

The first two moments are independent of $G$; however, this is not the case for higher moments (for example, in the third moment we have the possibility of a loop). For these higher moments, we need to perform an averaging over $G$ as well, and study \be \mu_{d,\w}(k) \ = \ \frac1{|\mathcal{R}_{N,d}|} \sum_{G\in\mathcal{R}_{N,d}} \mu_{d,\w}(k;G).\ee

While we can compute any $\mu_{d,\w}(k)$, the calculations quickly become very involved, and indicate the need for a unified approach if we desire a tractable closed form expression. For example, the average (over weights drawn from a fixed $\w$ and $G \in \mathcal{R}_{N,d}$) for the next two even moments are
\bea\label{eq:fourthandsixmoments} \mu_{d,\w}(4) & \ = \ & d \mu_\w(4) + 2d(d-1) \mu_\w(2)^2 \nonumber\\
\mu_{d,\w}(6) & \ = \ &	d \mu_\w(6) + 6d(d-1) \mu_\w(4) \mu_\w(2) + [3d(d-1)^2 + 2d(d-1)(d-2)] \mu_\w(2)^3,\nonumber\\
\eea
where as always $\mu_\w(k)$ denotes the $k^\text{th}$ moment of the weight distribution $\w$ (the odd moments are easily shown to vanish). We prove these formulas in Lemma \ref{lem:MomentsIdentity}.

Recall that our goal is to find a distribution $\w$ so that $T_d \w = \w^{(\lambda)}$ for some $\lambda$, normalized to have second moment equal to 1/4 (the second moment of the semi-circle). Our second moment calculation in Lemma \ref{lem:firsttwoments} suggests that $\lambda = \sqrt{d}$. If the semi-circle is a fixed eigendistribution, then we must have $\mu_{d,\w}(4) = d^2/8$ and $\mu_{d,\w}(6) = 5d^3/8$. From \eqref{eq:fourthandsixmoments}, we see that if we choose $\w$ so that the fourth moment is 1/8 then we do get $\mu_{d,\w}(4) = d^2/8$, and if the sixth moment of $\w$ is also $5/64$ then $\mu_{d,\w}(6) = 5d^3/64$. These results suggest that we can inductively show that the semi-circle is a fixed eigendistribution, but a more involved calculation (see Lemma \ref{lem:MomentsIdentity}) shows this breaks down at the eighth moment: \bea\label{eq:averageeighthmoment} \mu_{d,\w}(8) & \ = \ & d\mu_{\w}(8) + 8d(d-1)\mu_{\w}(6)\mu_{\w}(2) + 6d(d-1)\mu_{\w}(4)^2
\nonumber\\ & & \  +\ 16d(d-1)^2\mu_{\w}(4)\mu_{\w}(2)^2 + 12d(d-1)(d-2)\mu_{\w}(4)\mu_{\w}(2)^2  \nonumber\\ & & \ + \ 4d(d-1)^3\mu_{\w}(2)^4 + 8d(d-1)^2(d-2)\mu_{\w}(2)^4\nonumber\\ & & \ + \ 2d(d-1)(d-2)(d-3)\mu_{\w}(2)^4. \eea If $\w$ is an eigendistribution of $T_d$ with $\lambda = \sqrt{d}$ then $\mu_{d,\w}(8)$ must equal $d^4 \mu_\w(8)$, which from the above implies \be\mu_{\w}(8) \ = \ \frac{7}{128} + \frac{1}{128(d^2 + d + 1)}. \ee Note this is almost, but not quite, the eighth moment of the normalized semi-circle (which is 7/128).

To unify the derivation of \eqref{eq:fourthandsixmoments} and \eqref{eq:averageeighthmoment}, as well as the higher moments, we introduce some notation. This allows us to give a compact, tractable closed form expression for these moments, and helps us prove there is a unique eigendistribution (and determine its moments).

\subsection{Closed acyclic path patterns}\label{sect:capp}

From \eqref{eq:moment1} and \eqref{eq:moment2}, it is clear that moments of the spectral distribution are closely related to the set of closed walks in $G$. Moreover, we shall demonstrate below that it suffices to restrict our attention to walks containing no cycles, as all the walks with at least one closed cycle contribute a negligible amount to \eqref{eq:moment2}. We now introduce a combinatorial object which will keep track of all closed walks on a large tree.

\begin{defi}
A \emph{closed acyclic path pattern} (\capp) is a string of symbols such that
\begin{enumerate}
\item every symbol which appears at all appears an even number of times; and
\item in the substring of symbols between any two consecutive instances of the same symbol, every symbol which appears at all appears an even number of times.
\end{enumerate}
\end{defi}
We call two \capps\ \emph{equivalent} if they differ only by a relabeling of the symbols.
The following is the \emph{raison d'\^{e}tre} for our definition.

\begin{lem}[Classification of closed walks]\label{lem:cappclassification}
The closed acyclic path patterns classify the closed walks beginning at a given vertex in a large tree.
\end{lem}

\begin{proof} There is a natural map from the set of paths (closed or not) in a large tree to the set of sequences, where we treat the edges as symbols and just record the edges used in order.  It is evident that this map is ``injective'' (the relevant equivalence relations on paths and sequences coincide).  There are two issues. We must show
\begin{enumerate}
\item every closed path corresponds to a sequence which is a \capp; and
\item every \capp\ is realizable as the edge sequence of some path.
\end{enumerate}
These are not hard to see.  Removing any edge from a tree disconnects the tree into two connected components, so it makes sense to ask whether two vertices are on the ``same side'' of an edge or on ``opposite sides''.  Furthermore two vertices are on the same side of \emph{every} edge if and only if they are the same vertex.  If we follow a path in a tree, then the start and end points are on the same side of an edge if and only if we traverse that edge an even number of times.  By a straightforward induction on the length of the path/sequence, a sequence corresponds to an actual path in a tree if and only if the second condition in the definition of a \capp\ holds. Likewise, a path is closed if and only if the corresponding sequence satisfies the first condition in the definition of a \capp\ holds.
\end{proof}

We can now define the terms that will appear in Lemma \ref{lem:MomentsIdentity}, our closed form expression for the moments $\mu_{d,\w}(k)$. Given a \capp\ $\pi$, let $e_1, e_2, e_3, \ldots, e_r$ denote all the distinct symbols appearing in $\pi$, in order of appearance. Equivalently, the $e_i$ denote the edges composing the walk represented by $\pi$, ordered by first traversal.  We need the following definitions.\\

\begin{itemize}

\item We denote the set of (equivalence classes of) \capps\ of length $k$ by $\mathcal{P}_k$. Note $\mathcal{P}_k$ is empty for $k$ odd. For $\pi \in P_{2k}$, we define the \textbf{diagram} of $\pi$ to be the minimal ordered, rooted tree which is traversed by the path described by the pattern, with edges repeated according to how often the edge is traversed in each direction.

\item The \textbf{multiplicity} of a \capp\ $\pi$ is $m_\pi(d)$, where $m_\pi$ is the polynomial
\begin{equation}
\label{eq:multiplicity}
m_\pi(x)\ = \ \prod_{j=1}^r (x - \alpha_j),
\end{equation}
where $\alpha_j := \# \{i < j : e_i \text{ is adjacent to } e_j\}$; we call $\alpha_j$ the multiplicity of edge $e_j$. Note that $(d-\alpha_j)/d$ is the percentage of edges emanating from vertex $j$ that are not yet used in $\pi$ when vertex $j$ is first visited. This is used in calculating contributions to the moments, as $d-\alpha_j$ represents the number of possibilities available in choosing the next distinct edge. We measure adjacency by looking at the edges on the tree, \emph{not} by the ordering of the edges in our symbol. Thus if $\pi =abccbddbeeba$ the multiplicity of $a$ is 0, that of $b$ and of $c$ is 1, and that of $d$ and of $e$ is 2.  Figure \ref{fig:multiplicitydiagram} illustrates (in the case of a 4-regular graph) how the number of choices at each stage depends on the shape of the path so far.
\begin{figure}
\begin{center}
\scalebox{.437}{\includegraphics{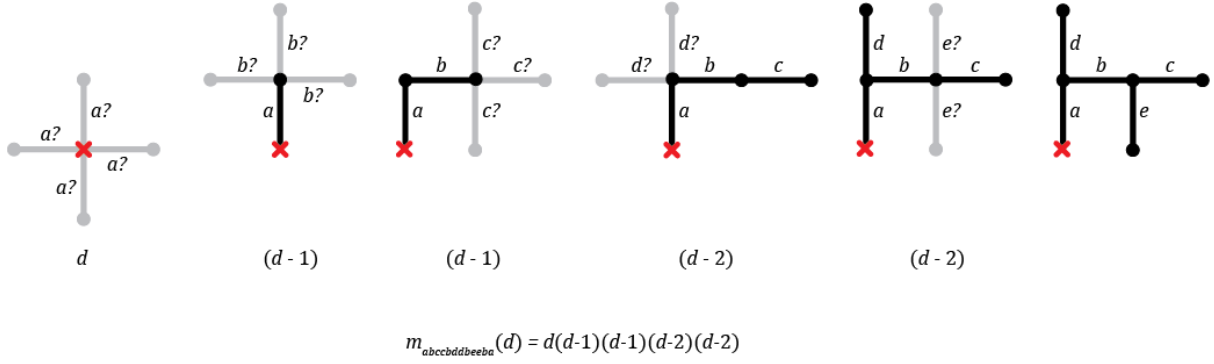}}
\caption{\label{fig:multiplicitydiagram} Choosing an realization of a particular \capp in a $d$-regular graph with $d=4$; this illustrates the multiplicity formula as an instance of standard counting principles.}
\end{center}\end{figure}


\item The \textbf{signature} of $\pi$ is
\begin{equation}
\label{eq:signature}
\sigma(\pi) \ := \  (n_1, n_2, \ldots, n_r),
\end{equation}
where $n_i$ denotes the number of times the symbol $e_i$ appears in $\pi$. Thus each $n_i$ is a positive integer. If $\pi \in \mathcal{P}_k$ then the sum of the entries of its signature is $k$.

\subitem $\diamond$: \textbf{$\mathcal{P}^{(2)}_k$} is the set of all \capps\ in $\mathcal{P}_k$ with signature $(2, 2, \dots, 2)$.

\subitem $\diamond$: \textbf{$\mathcal{P}^{(4)}_k$} is the set of all \capps\ in $\mathcal{P}_k$ with signature $(4, 2, \dots, 2)$.

\subitem $\diamond$: $\mathcal{P}^\circ_k$ is the set of all \capps\ in $\mathcal{P}_k$ excluding the pattern with signature $(k)$.

\item Given a signature $\sigma(\pi) = (n_1, n_2, \ldots, n_r)$ and a random variable $\w$, the \textbf{moment contribution associated to $\pi$ with respect to $\w$} is \be\label{eq:muwsigmaofpi} \mu_{\w}(\sigma(\pi)) \ := \ \mu_\w(n_1) \mu_\w(n_2) \cdots \mu_\w(n_r). \ee

\end{itemize}

\ \\

We can now give a complete description of the moments of the limiting spectral distribution (averaging over weights drawn from a fixed $\w$ and averaging over $G \in \mathcal{R}_{N,d}$ with $N\to\infty$). Our answer is in terms of the moments of the weight distribution $\w$ and some combinatorial data.

\begin{lem}[Moment Expansion]\label{lem:MomentsIdentity} Fix a weight $\w$ and a degree of regularity $d$. Let $\mu_\w(k)$ be the $k^\text{{\rm th}}$ moment of $\w$, $\mu_{d,\w}(k)$ the average over $G \in \mathcal{R}_{N,d}$ and over weights $w_{ij}$ drawn from $\w$ of the $k^\text{{\rm th}}$ moments of the measures $\nu_{d,G_{\textbf{w}}}$, and $\mathcal{P}_k$ the collection of all \capps\ of length $k$. For all natural numbers $k$ we have
\be\label{eq:momentexpansionformulafromthm}
\mu_{d,\w}(k) \ = \ 	\sum_{\pi \in \mathcal{P}_{k}} m_\pi(d) \mu_{\w}(\sigma(\pi)),
\ee
where $m_\pi(d)$, $\sigma(\pi)$ and $\mu_{\w}(\sigma(\pi))$ are defined in \eqref{eq:multiplicity} through \eqref{eq:muwsigmaofpi}.
\end{lem}
\noindent
Since $\mathcal{P}_k$ is trivially empty for all odd $k$, Lemma \ref{lem:MomentsIdentity} immediately implies
\begin{cor}
All odd moments vanish in the limit.
\end{cor}

\subsection{Proof of Lemma \ref{lem:MomentsIdentity}}\label{sec:proofofthmmomentsidentity}

Before proving Lemma \ref{lem:MomentsIdentity}, we show it is reasonable by deriving its prediction for the moment expansions of \eqref{eq:fourthandsixmoments} (we leave the eighth moment, \eqref{eq:averageeighthmoment}, to the reader). For the fourth moment, we need all \capps\ of length 4. We have \be \mathcal{P}_4 \ = \ \left\{ \pi_{4;1} = e_1 e_1 e_1 e_1,\ \pi_{4;2} = e_1 e_2 e_2 e_1, \ \pi_{4;3} = e_1 e_1 e_2 e_2\right\}. \ee The signatures are \be \sigma(\pi_{4;1}) \ = \ (4), \ \ \ \ \sigma(\pi_{4;2}) \ = \ (2,2), \ \ \ \ \sigma(\pi_{4;3}) \ = \ (2,2). \ee Recall the multiplicity $\alpha_j$ of $\pi$ is the number of $i < j$ such that $e_i$ is adjacent to $e_j$. We have \be m_{\pi_{4,1}}(d) \ = \ d-0, \ \ \ \ m_{\pi_{4;2}}(d) \ = \ (d-0)(d-1), \ \ \ \ m_{\pi_{4;3}}(d) \ = \ (d-0)(d-1). \ee Thus \be \sum_{\pi \in \mathcal{P}_4} m_\pi(d) \mu_{\w}(\sigma(\pi)) \ = \ d \mu_\w(4) + 2d(d-1)\mu_\w(2)\mu_\w(2), \ee in agreement with the first part of \eqref{eq:fourthandsixmoments}.

The calculation of the sixth moment is more involved, as we need to carefully determine the multiplicities. There are three cases. Note the sum of the entries of the signatures must equal 6, so there are only three possibilities: $(6)$, $(4,2)$, and $(2,2,2)$.

\begin{itemize}

\item Signature of $(6)$: The only $\pi$ that gives this is $e_1e_1e_1e_1e_1e_1$. The multiplicity is $d-0$, and the contribution is $d \mu_\w(6)$.

\item Signature of $(4,2)$: There are six possibilities: $e_1e_1e_1e_1e_2e_2$, $e_1e_1e_1e_2e_2e_1$, $e_1e_1e_2e_2e_1e_1$, $e_1e_2e_2e_1e_1e_1$, $e_1e_1e_2e_1e_1e_2$, and $e_1e_2e_1e_1e_2e_1$ (this last one is valid as the first and last $e_1$ are paired, and there are an even number of each symbol between them); all of these have signature $d(d-1)$, and the total contribution is thus $6d(d-1)\mu_\w(4)\mu_\w(2)$.

\item Signature of $(2,2,2)$: This is the first non-trivial case, as we have to carefully look and see where we are in our walk to determine the multiplicity. There are five terms. Three have multiplicity $d(d-1)^2$; they are $e_1e_1e_2e_3e_3e_2$, $e_1e_2e_3e_3e_2e_1$ and $e_1e_2e_2e_1e_3e_3$. Two have multiplicity $d(d-1)(d-2)$; they are $e_1e_1e_2e_2e_3e_3$ and $e_1e_2e_2e_3e_3e_1$. For example, for the last one we start at vertex 0 and move to vertex 1 by $e_1$, then to vertex 2 by $e_2$, then back to vertex 1 by $e_2$, then to vertex 3 by $e_3$, back to vertex 1 by $e_3$ and then return to vertex 0 by $e_1$. As all edges include vertex 1, they are all adjacent, thus $\alpha_1 = 0$, $\alpha_2 = 1$ and $\alpha_3 = 2$. The contribution from these five terms is $3d(d-1)^2 \mu_\w(2)^3 + 2d(d-1)(d-2)\mu_\w(2)^3$.

\end{itemize}

\ \\


We now turn to the proof of the Moment Expansion Lemma. We start with an informal discussion of the issues. We know that we can write the $n$\textsuperscript{th} spectral moment of a $d$-regular graph (not worrying yet about a limit along a sequence of graphs or averaging over the weights) as a sum of terms, where each term corresponds to a closed path in the graph of length $n$.  On the other hand, the summation in Lemma \ref{lem:MomentsIdentity} can also be thought of as a sum of similar terms if we interpret the summand $m_\pi(d) \mu_\w(\sigma(\pi))$ as $m_\pi(d)$ separate summands $\mu_\w(\sigma(\pi))$, one for each of the paths starting and ending at a given vertex with pattern $\pi$.  While these summations are similar, they are \emph{not} identical, since $G$ is not a tree but rather a specific $d$-regular graph which may or may not contain cycles. There are qualitatively different types of discrepancy here, both caused by small cycles: \\

\begin{itemize}

\item Paths which actually include a non-trivial cycle have no corresponding summand in our formula, since there is no path through a $d$-ary tree involving a cycle.\\

\item Paths which go partway around a cycle in both directions may have a corresponding summand in our formula, but their weights do not match.  For example, suppose there is a triangle with vertices $u$, $v$, $w$, where $u$ is the root.  Then the length 8 path $u$, $v$, $w$, $v$, $u$, $w$, $v$, $w$, $u$ uses edge $uv$ twice, edge $uw$ twice, and edge $vw$ four times, so this path contributes $\mu_\w(2)^2 \mu_\w(4)$ to the summation.  This path \emph{does} correspond to a term in our formula, $abbacddc$, but the signature is wrong.  That path contributes $\mu_\w(2)^4$ to our formula \eqref{eq:momentexpansionformulafromthm}.\\

\end{itemize}
The idea of the proof is to determine the contribution from a tree, and bound the average deviation of our $d$-regular graphs from being a tree. Although \eqref{eq:momentexpansionformulafromthm} does not give the correct spectral moments for individual graphs, it can give the correct limiting spectral moments for a sequence of graphs.  The technical condition that the number of small cycles in the graphs is growing slowly is precisely what is needed to guarantee that these discrepancies vanish in the limit. Fortunately there exist good bounds on the numbers of such small cycles in the family $\mathcal{R}_{N,d}$.




\begin{proof}[Proof of Lemma \ref{lem:MomentsIdentity}]  We first recall some notation. Given a $d$-regular graph $G$ on $N$ vertices (so $G \in \mathcal{R}_{N,d}$) and a probability distribution $\w$, we form the weighted graph $G_\textbf{w}$ whose edges are weighed by iidrv's drawn from $\w$. We denote average (with respect to the weights $w_{ij}$ being drawn from $\w$) of the $k^\text{th}$ moment of the associated spectral distributions $\nu_{d,G_{\textbf{w}}}$ by $\mu_{d,\w}(k;G)$. From \eqref{eq:moment1} and \eqref{eq:moment2} we know that $\mu_{d,\w}(k;G)$ is the average weight of a closed walk of length $k$ in $G$. The first step in the proof is to show that only acyclic walks contribute significantly to this average; i.e., all walks which contain cycles contribute a negligible amount.

We thus consider a closed path of length $k$, denoting the vertices by $i_1, i_2, \dots, i_k$. Let
\begin{equation}
\label{eq:SumWithCycles}
C_{d,G_\textbf{w}}(k) \ := \
\sum_{\substack{\langle i_1, i_2, \ldots, i_k, i_1 \rangle \\ \text{contains a cycle}}}
	b_{i_1 i_2} b_{i_2 i_3} \cdots b_{i_k i_1}
\end{equation} denote the contribution to the $k^\text{th}$ moment of $\nu_{d,G_{\textbf{w}}}$ from paths containing a cycle,
and
\be
A_{d,G_\textbf{w}}(k) \ := \
\sum_{\substack{\langle i_1, i_2, \ldots, i_k, i_1 \rangle \\ \text{contains no cycles}}}
	b_{i_1 i_2} b_{i_2 i_3} \cdots b_{i_k i_1}
\ee the contribution from the acyclic closed paths. We may thus rewrite equations \eqref{eq:moment1} and \eqref{eq:moment2} as
\begin{equation}
\label{star}
\mu_{d,G_{\textbf{w}}}(k) \ = \ \frac{1}{N} C_{d,G_\textbf{w}}(k) +
	\frac{1}{N} A_{d,G_\textbf{w}}(k).
\end{equation}

We will show that the first term tends to 0 as $N \to \infty$. This in turn implies that $\mu_{d,\w}(k)$ only depends on paths with no cycles (i.e., \capp s).
From Lemma \ref{lem:firsttwoments} we may assume $k \ge 3$.
Fix a $G\in \mathcal{R}_{N,d}$ and a weight vector $\textbf{w}$ with components independently drawn from $\w$. We may take all but $Nd/2$ of the entries of $\textbf{w}$ to be $0$ without affecting the weighted adjacency matrix; for notational convenience we label those weights which aren't necessarily $0$ by $\{w_1,w_2,\ldots,w_s\}$ (where $s=Nd/2$).

Choose some $k$-cycle in $G$; as it can only traverse these weighted edges, its contribution is $w_1^{r_1} w_2^{r_2} \cdots w_s^{r_s}$, where $r_i \geq 0$ and $\sum r_i = k$. Here $r_i$ represents the total number of times our $k$-cycle traverses the edge with weight $w_i$. Averaging over $\w$ and using the independence of the weights, we have that the expected contribution of a $k$-cycle is
\be
\begin{split}
\E[w_1^{r_1}\cdots w_s^{r_s}]
&\ = \
\E[w_1^{r_1}] \cdots \E[w_s^{r_s}] \\
&\ = \
\mu_\w(r_1) \mu_\w(r_2) \cdots \mu_\w(r_s) \\
&\ = \
\mu_\w(1)^{\alpha_1} \mu_\w(2)^{\alpha_2} \cdots \mu_\w(s)^{\alpha_s}
\end{split}
\ee
for some non-negative integers $\alpha_i$ satisfying $\sum_{i=1}^s \alpha_i = k$. This immediately implies that
\be
\alpha_{k+1}
\ = \  \alpha_{k+2}
\ = \  \cdots
\ = \  \alpha_s
\ = \  0,
\ee
whence
\be
\E[w_1^{r_1} \cdots w_s^{r_s}]
\ = \  \mu_\w(1)^{\alpha_1} \cdots \mu_\w(k)^{\alpha_k}.
\ee

Let \be M\ = \ \max \left\{\mu_\w(1)^{\alpha_1} \mu_\w(2)^{\alpha_2} \cdots \mu_\w(k)^{\alpha_k} : \alpha_i\geq 0, \sum_{i=1}^{k} i\alpha_i = k\right\}. \ee  Note that $M$ depends only $\{\mu_\w(i)\}_{i=1}^k$ (the first $k$ moments of $\w$) and $k$; in particular, it is bounded independent of $N$. We highlight this fact by writing $M = M(\w,k)$.

Define $C_{G,i}$ to be the total number of $i$-cycles in $G$. For a fixed weight distribution $\w$, the contribution of the paths with cycles to $\mu_{d,\w}(k)$ from averaging over weights drawn from $\w$ and graphs $G \in \mathcal{R}_{N,d}$ is
\bea & &
\frac1{|\mathcal{R}_{N,d}|} \sum_{G\in\mathcal{R}_{N,d}} \intii \cdots \intii \frac{1}{N} C_{d,G_\textbf{w}}(k) \prod_{1 \le i \le j \le N} p_\w(w_{ij}) {\rm d}w_{ij}
\nonumber\\ & & \ \ \ \ \ \le \
\frac{1}{|\mathcal{R}_{N,d}|} \sum_{G \in \mathcal{R}_{N,d}}
	\left(\frac{1}{N}\sum_{i = 3}^k C_{G,i} M(\w,k) \right)
\nonumber\\ & & \ \ \ \ \ = \
\frac{M(\w,k)}{N} \cdot \sum_{i = 3}^k \frac{1}{|\mathcal{R}_{N,d}|}
	\sum_{G \in \mathcal{R}_{N,d}} C_{G,i}
\nonumber\\ & & \ \ \ \ \ = \
\frac{M(\w,k)}{N} \cdot \sum_{i = 3}^k C_{i,N,d}.
\eea
\noindent
Here $i \geq 3$, since otherwise $G$ has no cycles.
By Lemma 4.1 of \cite{McK}, we know that for $i \ge 3$ we have
\be
\lim_{N\to\infty} C_{i,N,d}\ = \ \frac{(d-1)^i}{2i}.
\ee
Combining this with the above, we deduce that the contribution from the paths with a cycle to $\mu_{d,\w}(k)$ is $O(1/N)$, and thus negligible as $N\to\infty$. In particular, this implies \be \mu_{d,\w}(k) \ = \ \lim_{N\to\infty} \frac1{|\mathcal{R}_{N,d}|} \intii \cdots \intii \frac{1}{N} A_{d,G_\textbf{w}}(k) \prod_{1 \le i \le j \le N} p_\w(w_{ij}) {\rm d}w_{ij}. \ee

The proof is completed by noting that this is equivalent to $\sum_{\pi \in \mathcal{P}_k} m_\pi(d) \mu_\w(\sigma(\pi))$. This follows from the definition of \capps, multiplicities and signatures, and similar arguments as in \cite{McK}. The factor $\mu_\w(\sigma(\pi))$ is clear, arising from how often each weight occurs and then averaging over the weights.

The factor $m_\pi(d)$ requires a bit more work. As we take the limit as $N\to\infty$, there is no loss in assuming we have a tree. We must have a closed path in order to have a contribution, and thus $k$ must be even and each edge must be traversed an even number of times (as we have a tree, there are no cycles). By Lemma \ref{lem:cappclassification} there is a one-to-one correspondence between \capps\ and legal walks along edges. Each time we hit a vertex and go off along a new edge, the number of choices we have equals $d$ (the regularity degree) minus the number of edges we have already taken from the vertex. This is why the multiplicity of edge $e_j$ is $d$ minus the number of edges $e_i$ adjacent to $e_j$ with $i < j$, and adjacency is measured relative to the tree and \emph{not} the string of edges. This completes the proof. \end{proof}


While Lemma \ref{lem:MomentsIdentity} gives a closed form expression for the limiting moments, it is not immediately apparent that it is a \emph{useful} expansion. We need a way of computing the sum over $\pi$ of $m_\pi(d) \mu_\w(\sigma(\pi))$, which 
is our next subject.




%

\subsection{Counting walks by signature}\label{sec:countingwalksbysig}

We conclude this section by showing how to count walks with certain simple signatures. We use these results to prove our theorems on eigendistributions in \S\ref{sec:eigendistr}. We first recall some notation. \\

\begin{itemize}

\item $\mathcal{P}^{(2)}_k$ is the set of all \capps\ in $\mathcal{P}_k$ with signature $(2,2,2,\ldots,2)$.

\item $\mathcal{P}^{(4)}_k$ is the set of all \capps\ in $\mathcal{P}_k$ with signature $(4,2,2,\ldots, 2)$.

\item $\mathcal{T}_k$ is the set of all triples $(\pi,x,y)$, where $\pi\in \mathcal{P}^{(2)}_k$ and $x$, $y$ are symbols corresponding to distinguished edges in the diagram, which must be adjacent and first traversed in that order.

\item $\mathcal{P}^\circ_k$ is the set of all \capps\ in $\mathcal{P}_k$ excluding the pattern with signature $(k)$.

\end{itemize}

\ \\

\begin{lem}[Counting Walks without Repeated Edges]
\label{lem:CountingWalks}
There are exactly $\frac{1}{k+1}\binom{2k}{k}$ \capps\ of length $2k$ and signature $(2,2,2,\ldots,2)$.  That is,
$|\mathcal{P}^{(2)}_{2k}|=\frac{1}{k+1}\binom{2k}{k}$.
\end{lem}

\begin{proof}
Walks of signature $(2,2,2,\ldots,2)$ use each edge exactly twice. Such a walk is determined by its diagram, regarded as an ordered tree (in the sense that the children of each vertex ``remember'' in which order they were visited).  It is well-known that the Catalan numbers count such trees, and appear throughout random matrix theory (see for example \cite{AGZ}).
\end{proof}

The following lemma plays a key role in computing the lower order term to the moments in Theorem \ref{thm:mainresultuniquewdmoments}, and allows us to improve our error from $O(1/d)$ to $O(1/d^2)$.

\begin{lem}[Serendipitous Correspondence]
\label{lem:Serendipity}
There is a two-to-one correspondence between length-$2k$ \capp s whose signature is $(2,2,2,\ldots,2)$ with a distinguished pair of adjacent edges, and length-$2k$ \capps\ with signature $(4,2,2,\ldots,2)$. That is,
$|\mathcal{T}_{2k}|=2|\mathcal{P}^{(4)}_{2k}|$.
\end{lem}

\begin{proof}
In this proof, we always use the symbols $x,y$ (in that order) as the distinguished symbols for an object in $\mathcal{T}_{2k}$.  These symbols will occur either in the order $xyyx$ or $xxyy$ (the order cannot be $xyxy$ or the \capp\ condition would be violated).  Consider sequences $AxByCyDxE$ (case 1) or $AxBxCyDyE$ (case 2), where capital letters denote substrings, where every symbol occurring in ABCDE does so exactly twice in total.  In order for this to be a genuine pattern, for each non-distinguished symbol that occurs, one of the following must be true.

\begin{enumerate}

\item Both occurrences are in the same substring ($A$, $B$, $C$, $D$, or $E$).

\item One occurrence is in $A$ and the other is in $E$.

\item One occurrence is in $B$ and the other is in $D$ (case 1 only).

\item One occurrence is in $C$ and the other is in $A$ or $E$ (case 2 only).

\end{enumerate}

Since $x$ and $y$ are adjacent edges, the last two possibilities are ruled out.  But now it is not hard to see that elements of $\mathcal{P}^{(4)}_{2k}$ have the form $AzBzCzDzE$, with precisely the same conditions on $A,B,C,D,E$.  Then we correspond the patterns $AxByCyDxE$ and $AxBxCyDyE$ to the pattern $AzBzCzDzE$, giving the desired two-to-one correspondence.
\end{proof}

\section{The eigendistribution}\label{sec:eigendistr}

Our goal is to find the eigendistributions $\w^{(\lambda)}$ of the maps $T_d$ from \eqref{eq:DefnOfTd}. Recall that $T_d$ maps a given weight distribution $\w$ to a spectral distribution. In this section we prove that for each $d$ there exists a unique (up to rescaling) eigendistribution of $T_d$. To do this, we first apply Lemma \ref{lem:MomentsIdentity} to obtain a recursive identity on the moments of any eigendistribution; it will then be seen that there exists a distribution possessing these moments. Moreover, we show that after appropriate rescaling, the moments grow very similarly to those of the semicircle distribution. The two distributions are not exactly the same, and we quantify the extent to which they differ. \\

We first demonstrate that for any fixed $d$, $T_d$ has at most one eigenvalue. Given any distribution $\w$ with density $p_\w$ and any $\lambda > 0$, let $\mu_\w(k)$ denote the $k^\text{th}$ moment of $\w$ and $\mu_{\w^{(\lambda)}}(k)$ denote the $k^\text{th}$ moment of the rescaled distribution $\w^{(\lambda)}$ (see \eqref{eq:effectoflambdaonmeasure} for the effect of scaling a measure by $\lambda$). We have
\be
\mu_{\w^{(\lambda)}}(k)\ = \ \intii x^k \, d\w^{(\lambda)}(x)\ = \ \intii x^k \lambda \, p_\w(\lambda x) \, {\rm d}x\ = \ \intii \left(\frac{x}{\lambda}\right)^k  p_\w(x) \, {\rm d}x\ = \ \frac{1}{\lambda^k} \mu_\w(k).
\ee
In particular, if $\w$ is an eigendistribution with eigenvalue $\lambda$, and $\mu_{d,\w}(k)$ denotes the moments of the spectral distribution $T_d \w = \w^{(\lambda)}$, then $\mu_{d,\w}(2) = \lambda^{-2} \mu_\w(2)$. On the other hand, from Lemma \ref{lem:firsttwoments} we know that $d^{-1}\mu_{d,\w}(2) = \mu_\w(2)$, whence
\be
\lambda\ = \ d^{-1/2}.
\ee
We thus obtain a relation for the even moments of an eigendistribution:
\begin{equation}
\label{eq:SpectralMoments}
\mu_{d,\w}(2k)\ = \ d^k \mu_\w(2k).
\end{equation}
Substituting this into Lemma \ref{lem:MomentsIdentity} and simplifying yields the following formula.

\begin{lem}[Eigenmoment Formulas]
\label{lem:Eigenmoments}
Suppose $\w_d$ is an eigendistribution of $T_d$,
i.e., $T_d\w_d = \w_d^{(\lambda)}$ for some $\lambda > 0$. Denote the moments of $\w_d$ by $\mu_{\w_d}(k)$. We may assume (without loss of generality) that $\w_d$ is scaled so that $\mu_{\w_d}(2) = 1/4$ (the second moment of the normalized semi-circle distribution). Then $\mu_{\w_d}(k) = 0$ for all odd $k$, and
\begin{equation}
\label{eq:Eigenmoments}
\muwd(2k) \ = \ \frac{1}{d^k - d}
	\sum_{\pi \in \mathcal{P}^\circ_{2k}} m_\pi(d) \muwd(\sigma(\pi)).
\end{equation}
\end{lem}

We can now prove Theorem \ref{thm:mainresultuniquewdmoments}, namely that there exists a unique eigendistribution, as well as determine properties of its moments.

\begin{proof}[Proof of Theorem \ref{thm:mainresultuniquewdmoments}] Since the signature $\sigma(\pi)$ consists of numbers strictly smaller than $2k$ for all $\pi \in \mathcal{P}^\circ_{2k}$, \eqref{eq:Eigenmoments} gives a recursive formula for the moments. Thus, if an eigendistribution $\w_d$ exists, then its moments are uniquely specified. The even moments are easily seen to be bounded above by 1 and below by the moments of the normalized semi-circular distribution. Thus Carleman's condition is satisfied ($\sum_{k=1}^\infty \muwd(2k)^{-1/2k} = \infty$), and the moments uniquely determine a distribution (see for example \cite{Bi, Ta}).

Lemma \ref{lem:Eigenmoments} can be used to calculate small moments with relative ease: $\muwd(2) = 1/4$,
$\muwd(4) = 1/8$, $\muwd(6) = 5/64$, and \be
\muwd(8)\ = \ \frac{7}{128} + \frac{1}{128(d^2 + d + 1)}.
\ee
From this data it seems safe to guess that the main term of $\muwd(2k)$ is
\be
c_{2k}\ := \ \frac{1}{4^k (k+1)} \binom{2k}{k}
\ee
(the $2k^\text{th}$ moment of the normalized semi-circular distribution), which we now prove. We first show that
\be
\muwd(k)\ = \ c_{k} + O(1/d),
\ee
and then with a bit more work improve the error to $O(1/d^2)$.

For odd $k$ there is nothing to prove, since both $\muwd(k)$ and $c_k$ vanish. We thus restrict our attention to even $k$, and proceed by induction. For $2k \le 8$, we have already verified the conjecture. The only role of the inductive hypothesis is to ensure that, when computing $\muwd(2k)$, we can treat all lower eigenmoments as $O(1)$. The recursion formula \eqref{eq:Eigenmoments} gives
\be
(d^k - d) \muwd(2k)
\ =\
	\sum_{\pi \in \mathcal{P}^\circ_{2k}} m_\pi(d) \muwd(\sigma(\pi)).
\ee
The total contribution from those $\pi$ which involve fewer than $k$ symbols is $O(d^{k-1})$. Thus, the main term must come from the patterns involving $k$ edges, i.e., $\pi$ whose signature $\sigma(\pi) = (2,2,\ldots,2)$. Recall $\mathcal{P}^{(2)}_k$ is the set of all \capps\ of length $k$ which possess a signature of this form. We have
\bea
(d^k - d) \muwd(2k)
&\ = \ & \sum_{\pi \in \mathcal{P}^{(2)}_{2k}} m_\pi(d) \muwd(\sigma(\pi)) + O(d^{k-1}) \nonumber\\
&\ = \ & \left(\sum_{\pi \in \mathcal{P}^{(2)}_{2k}} \muwd(\sigma(\pi))\right) d^k + O(d^{k-1}) \nonumber\\
&\ = \ & |\mathcal{P}^{(2)}_{2k}| 2^{-2k} d^k + O(d^{k-1}).
\eea
Lemma \ref{lem:CountingWalks} yields the desired conclusion.

By using the serendipitous correspondence described in Lemma \ref{lem:Serendipity}, we can sharpen the error term and obtain
\be
\muwd(2k)\ = \ c_{2k} + O(1/d^2).
\ee
\noindent
As above, we have already verified the theorem (with no error term) in the cases when $k$ is odd or at most 8. Henceforth we assume that $2k > 8$.  We analyze the contribution from patterns with at least $k-1$ distinct symbols (in other words, we allow at most one repetition), and trivially bound the contribution from the remaining by $O(d^{k-2})$. Note $\muwd(2) = 1/4$ and $\muwd(4) = 1/8$, so if $\pi \in \mathcal{P}^{(2)}_{2k}$ then $\muwd(\sigma(\pi)) = (1/4)^k$, while if $\pi \in \mathcal{P}^{(4)}_{2k}$ then $\muwd(\sigma(\pi)) = (1/8)(1/4)^{k-1}$. We have
\bea\label{eq:dkminusdmuwd2k}
(d^k - d) \muwd(2k)
&\ = \ & \sum_{\pi \in \mathcal{P}^\circ_{2k}} m_\pi(d) \muwd(\sigma(\pi)) \nonumber\\
&\ = \ &  \left(\sum_{\pi \in (\mathcal{P}^{(2)}_{2k} \cup \mathcal{P}^{(4)}_{2k})} m_\pi(d) \muwd(\sigma(\pi))\right) + O(d^{k-2}) \nonumber\\
&\ = \ & (1/4)^{k}\left(\sum_{\pi \in \mathcal{P}^{(2)}_{2k}} m_\pi(d)\right)+(1/8)(1/4)^{k-1}\left(\sum_{\pi\in \mathcal{P}^{(4)}_{2k}} m_\pi(d)\right)+O(d^{k-2})\nonumber\\
&\ = \ & (1/4)^{k}\left(\sum_{\pi\in \mathcal{P}^{(2)}_{2k}} m_\pi(d)+2\sum_{\pi \in \mathcal{P}^{(4)}_{2k}} m_\pi(d)\right)+O(d^{k-2}).
\eea
The strategy is to compute the secondary terms of $m_\pi(d)$, multiply them by the correct factor and then substitute back into \eqref{eq:dkminusdmuwd2k}.  If $\pi \in \mathcal{P}^{(2)}_{2k}$, then
\be
m_\pi(d)
\ =\ \prod_{i=1}^k (d - \alpha_i)
\ =\ d^k - \left(\sum_{i=1}^k \alpha_i \right) d^{k-1} + O(d^{k-2}),
\ee
where $\alpha_i$ is the number of edges prior to the $i^\text{th}$ which are adjacent to the $i^\text{th}$ as in  \eqref{eq:multiplicity}. Summing over $i$ gives the total number of pairs of adjacent edges in the diagram. Summing over $\mathcal{P}^{(2)}_{2k}$, we obtain
\be
\sum_{\pi \in \mathcal{P}^{(2)}_{2k}} m_\pi(d)\ = \ |\mathcal{P}^{(2)}_{2k}| d^k - |\mathcal{T}_{2k}| d^{k-1} + O(d^{k-2}).
\ee All of these terms have the same value for $\muwd(\sigma(\pi))$, namely $(1/4)^k$.

For the other summation we need only the dominant term:
\be
\sum_{\pi \in \mathcal{P}^{(4)}_{2k}} m_\pi(d)\ = \ |\mathcal{P}^{(4)}_{2k}| d^{k-1} + O(d^{k-2}).
\ee All of these terms have the same value for $\muwd(\sigma(\pi))$, namely $(1/8) (1/4)^{k-1} = (1/2) (1/4)^k$.

Using the above, we find that the contribution from $\pi \in \mathcal{P}_{2k}^{(2)} \cup \mathcal{P}_{2k}^{(4)}$ to \eqref{eq:dkminusdmuwd2k} is
\[
|\mathcal{P}^{(2)}_{2k}| \frac{d^k}{4^k} - |\mathcal{T}_{2k}| \frac{d^{k-1}}{4^k} + |\mathcal{P}^{(4)}_{2k}| \frac{d^{k-1}}{2 \cdot 4^k} + O\left((d/4)^{k-2}\right);
\]
however, by Lemma \ref{lem:Serendipity} we have $|\mathcal{P}^{(4)}_{2k}| = 2 |\mathcal{T}_{2k}|$, and thus the order $d^{k-1}$ terms above cancel, yielding \be \mu_{\mathcal{W}_d}(2k) \ = \ |\mathcal{P}^{(2)}_{2k}| \frac{d^k}{4^k (d^k-d)} + O\left(\frac{d^{k-2}}{4^k (d^k-d)}\right). \ee From Lemma \ref{lem:CountingWalks} we have $|\mathcal{P}^{(2)}_{2k}| = \frac1{k+1}\binom{2k}{k} = 4^k c_{2k}$ (where $c_{2k}$ is the $2k$\textsuperscript{th} moment of the semi-circle distribution normalized to have variance 1/4), and thus \be \mu_{\mathcal{W}_d}(2k) \ = \ c_{2k} + O\left(\frac{c_{2k}}{d^{k-1}}\right) + O\left(\frac{1}{4^k d^2}\right); \ee as $k > 3$ the second error term dominates. We conclude that the error is $O(1/d^2)$, as claimed.
\end{proof}

\ \\

\end{document}